\documentclass[3p]{elsarticle}

\usepackage[dvips]{hyperref}
\usepackage{breakurl}
\usepackage{amssymb,amsmath}
\usepackage{amsthm}
\usepackage{psfrag}
\usepackage{enumitem}

\newcommand{\R}{\mathbb R}

\newcommand{\N}{\mathbb N}

\newcommand{\M}{\mathcal   M}
\newcommand{\K}{\mathcal K}
\newcommand{\Ho}{\mathcal{H}}
\newcommand{\Ro}{\mathcal R}
\newcommand{\Uo}{\mathcal{U}}
\newcommand{\Vo}{\mathcal{V}}

\newcommand{\rmd}{\mathrm d}
\newcommand{\coloneqq}{:=}
\newcommand{\edot}{\,\cdot\,}

\newcommand\abs[1]{\left\vert#1\right\vert}
\newcommand\sabs[1]{\lvert#1\rvert}

\newcommand\set[1]{\left\{#1\right\}}
\newcommand\mset[1]{\bigl\{#1\bigr\}}
\newcommand\sset[1]{\{#1\}}
\newcommand{\Om}{\Omega}
\newcommand{\om}{\omega}

\newcommand{\kl}[1]{\left(#1\right)}

\newcommand{\skl}[1]{(#1)}

\newcommand{\req}[1]{(\ref{eq:#1})}

\newtheorem{theorem}{Theorem}[section]
\newtheorem{lemma}[theorem]{Lemma}

\newtheorem{remark}[theorem]{Remark}

\numberwithin{equation}{section}
\numberwithin{figure}{section}
\numberwithin{theorem}{section}

\begin{document}

\begin{frontmatter}

\title{Inversion of  circular means  and the wave equation\\
on convex planar domains}

\author{Markus Haltmeier}

\address{Department of Mathematics\\
University of Innsbruck\\
Technikestra{\ss}e 21a, A-6020 Innsbruck\\
E-mail:
\href{mailto:markus.haltmeier@uibk.ac.at}
{\tt markus.haltmeier@uibk.ac.at}}

\begin{abstract}
We  study the problem of recovering the initial data
of the two dimensional wave equation  from values of its  solution on
the boundary $\partial \Om$ of a  smooth convex bounded domain $\Om \subset \R^2$.
As a main result we establish back-projection type inversion formulas
that  recover any initial data with support in $\Om$ modulo an
explicitly computed  smoothing integral operator $\K_\Om$.
For circular and elliptical domains the operator $\K_\Om$  is shown to
vanish identically and hence we establish exact inversion
formulas of the back-projection type in these cases.
Similar results are  obtained for recovering  a function
from its mean values over circles with centers on $\partial \Om$.
Both reconstruction problems  are, amongst others, essential
for the hybrid imaging modalities photoacoustic and thermoacoustic
tomography.
\end{abstract}

\begin{keyword}
Circular means \sep
spherical means \sep
reconstruction formula \sep
photoacoustic imaging \sep
thermoacoustic tomography \sep
Radon transform.

\medskip
\MSC
35L05 \sep
44A12 \sep
65R32 \sep
92C55.
\end{keyword}

\end{frontmatter}

\section{Introduction}
\label{sec:intro}

Let $\Om \subset \R^2$ be a convex bounded domain in the plane with smooth boundary $\partial \Om$.
Suppose that $f \colon  \R^2\to \R $ is a smooth function that vanishes outside $\Om$,
and denote by
$u\colon \R^2 \times \kl{0, \infty} \to \R$
the solution of the initial value problem
\begin{equation}  \label{eq:wave}
	\left\{
	\begin{aligned}
	\kl{\partial_{t} ^2  - \Delta_{x} }
	u\kl{x, t}
	&=
	0
	& \text{ for }
	\kl{x,t} \in
	\R^2 \times \kl{0, \infty}
	\\
	u\kl{x,0}
	&=
	f\kl{x}
	& \text{ for }
	x  \in \R^2
	\\
	\kl{\partial_{t} u}
	\kl{x,0}
	&=0
	&
	\text{ for }
	x  \in \R^2
	\end{aligned}
	\right.
	\,.
\end{equation}
In this paper we study the problem of recovering the unknown
initial data $f$  from the values  of the solution $u$ on the observation surface $\partial \Om \times \kl{0,\infty}$.
Because  the  solution of~\req{wave} can be  expressed in terms of
circular means (and vice versa), the  inversion of the wave equation is basically
equivalent to the problem of recovering the function $f$ from its averages over
circles with centers on $\partial \Om$.
Both reconstruction problems are essential for the novel  hybrid
imaging methods photoacoustic tomography (PAT)  and
thermoacoustic tomography (TAT).  The  standard setups in  PAT/TAT
using point-like detectors require the inversion of the wave equation in three spatial dimensions~\cite{KucKun08,XuWan06}.
The two dimensional version considered in this paper
arises in  a variant  of PAT/TAT that uses  linear integrating
detectors instead of point-like ones
(see~\cite{BurBauGruHalPal07,PalNusHalBur07b}).
Inversion of circular means and the  wave equation
is also relevant for other imaging modalities, such as
SONAR \cite{And88} or  ultrasound tomography  \cite{NorLin81}.

In the recent years, many reconstruction techniques for the wave inversion and the
inversion from circular  means (or spherical means in higher dimensions)
have been  developed.
These techniques can be classified in iterative reconstruction methods
(see \cite{PalNusHalBur07b,DeaNtzRaz12,DonGoeKun11,PalViaPraJac02,ZhaAnaPanWan05}),
model based time reversal (see \cite{BurMatHalPal07,FinPatRak04,HriKucNgu08,SteUhl09}),  Fourier domain algorithms
(see \cite{NorLin81,Hal09,HalSchBurNusPal07,HalSchZan09b,KoeFraNiePalWebFre01,XuXuWan02}), and algorithms based
on explicit reconstruction formulas  of the  back-projection
type (see \cite{And88,FinPatRak04,Faw85,FinHalRak07,Hal11b,Kun07a,Kun11,Nat12,Pal12,XuWan05}).
The back-projection approach is  particularly appealing
since it is theoretically exact, stable
with respect to data  and modeling imperfections, mathematically elegant,   and quite straightforward to
implement numerically. Until  recently, however, exact back-projection type reconstruction
formulas were only known for the  cases where $\partial\Omega$ is a  spherical, planar or cylindrical surface (in three spatial dimensions)  and a circle or a line  (in two  spatial dimensions). Many researchers even believed
that exact reconstruction  formulas of the back-projection type exist only in such cases.
Only very recently  the so called universal back-projection formula
(originally  introduced to PAT/TAT  in \cite{XuWan05})
has been shown in \cite{Nat12}  to provide theoretically exact reconstruction for ellipsoids in $\R^3$.
In the same  paper it has  further be shown that for general convex
domains in $\R^3$ the universal back-projection formula
is exact modulo an explicitly given smoothing integral operator.

In this paper we establish back-projection type reconstruction
formulas for the inversion of the two dimensional wave equation
on convex domains $\Om \subset \R^2$. As their counterparts in three spatial dimensions,
the  derived formulas provide exact reconstruction  modulo an explicitly
computed smoothing integral operator $\K_\Om$.
We further show that for elliptical and circular  domains
the  operator $\K_\Om$ vanishes and
therefore we obtain exact back-projection type reconstruction
formulas in these cases. The same type of results are derived for reconstructing a function from its mean values over circles with centers on the boundary of a convex domain in the plane.
Note that our  reconstruction formulas as well as the  proofs
differ from the ones given in  \cite{Nat12}.
This  partially accounts for the fact that  the three dimensional wave
equation satisfies a  strong form of Huygens principle, whereas the
two  dimensional wave equation does not.

\subsection{Notation}
\label{sec:not}

Before presenting our main results we introduce some  notation that
will be used throughout this article.

For any $k \in \N \cup \set{\infty}$, we denote by $C_c^k\kl{\Om}$  the set of all
$k$-times continuously differentiable functions $f \colon \R^2 \to \R$  which have support in $\Om$.
For given   $f \in C_c^\infty  \kl{\Om}$ we denote  by
$\Uo f \colon \R^2 \times \kl{0, \infty} \to \R$ the solution of  the wave equation \req{wave} with initial data
$f$. Further, for  some integrable function $f \colon \R^2 \to \R$, the circular mean transform $\M f\colon \R^2 \times \kl{0,\infty} \to \R$ is defined by
\begin{equation} \label{eq:means}
	\kl{\M f} \kl{x,r}
	:= \frac{1}{2 \pi}
	\int_{S^1} f \kl{x + r \omega}
	\rmd \om
	\qquad
	\text{ for }
	\kl{x, r}  \in \R^2 \times \kl{0, \infty} \,.
\end{equation}
Here and below $S^1 \coloneqq \set{x \in \R^2: \abs{x} = 1}$ denotes the unit sphere in $\R^2$ with $\lvert x \rvert $ being   the   Euclidian norm of
$x \in \R^2$. The corresponding inner product of two points $x_0, x_1 \in \R^2$ will be denoted by
$x_0 \cdot x_1$. Similarly, the (classical) Radon transform $\Ro f \colon S^1 \times \R \to \R$
is defined by
\begin{equation*} 
\kl{\Ro f} \kl{n,a}
:=
\int_{\R}  f\kl{ a n  +  b n^\bot }
\rmd b
\qquad
\text{ for  }  \; \kl{n, a} \in S^{1} \times \R \,,
\end{equation*}
where $ n^\bot \in S^1 $ denotes  a unit vector orthogonal to $n$.
The derivative of a function  $\phi \colon S^{1} \times \R \to \R$
with respect to the   second variable  will be  denoted by
$\kl{\partial_a \phi} \kl{n, a}$, and   $\kl{\Ho_a  \phi} \kl{n, a} $
will be used to denote the Hilbert transform of $\phi$ in the second
argument, defined  as the distributional convolution with $1/\skl{\pi a}$.

We further write  $\chi_\Om \colon \R^2 \to \R$ for  the  characteristic function
of the domain $\Om$  (taking the value  one inside $\Om$  and zero outside) and set
\begin{equation} \label{eq:hatns}
\hat n \kl{x_1,x_0}
\coloneqq  \frac{x_1-x_0}{\abs{x_1-x_0}}
\,, \;\;
\hat a \kl{x_1,x_0}
\coloneqq
\frac{1}{2} \; \frac{\abs{x_1}^2 - \abs{x_0}^2}{\abs{x_1-x_0}}
\qquad \text{ for  }
\;
x_1 \neq x_0 \in \R^2 \,.
\end{equation}
Finally,  for any $f \in C_c^\infty  \kl{\Om}$ and any $x_0 \in \Om$, we define
\begin{equation}\label{eq:kern}
	\kl{\K_\Om f} \kl{x_0}
	\coloneqq
	\frac{1}{8\pi}
	\int_{\Omega}
	f\kl{x_1}  \frac{ \kl{ \partial_a^2 \Ho_a \Ro \chi_\Om} 	\kl{\hat n\kl{x_1,x_0},  \hat a\kl{x_1,x_0}}}
	{\abs{x_1-x_0}}  \, \rmd x_1 \,.
\end{equation}
It can be readily verified that the line
$\ell\kl{x_1,x_0} = \sset{ x \in \R^2 : \hat n\kl{x_1,x_0} \cdot x   =\hat a\kl{x_1,x_0}}$ consists of all points
having the same distance between $x_0$ and
$x_1$; compare with Figure~\ref{fig:geometry}.

\begin{psfrags}
\psfrag{a}{\scriptsize $x_0$}
\psfrag{b}{\scriptsize $x_1$}
\psfrag{s}{\scriptsize $\hat a \kl{x_1,x_0}$}
\psfrag{O}{\scriptsize domain $\Omega$}
\psfrag{S}{\scriptsize boundary  $\partial \Omega$}
\psfrag{m}{\scriptsize$\kl{x_1+x_0}/2$}
\psfrag{E}{\scriptsize$ \ell\kl{x_1,x_0}$}
\begin{figure}
\centering
\includegraphics[width=0.5\textwidth]{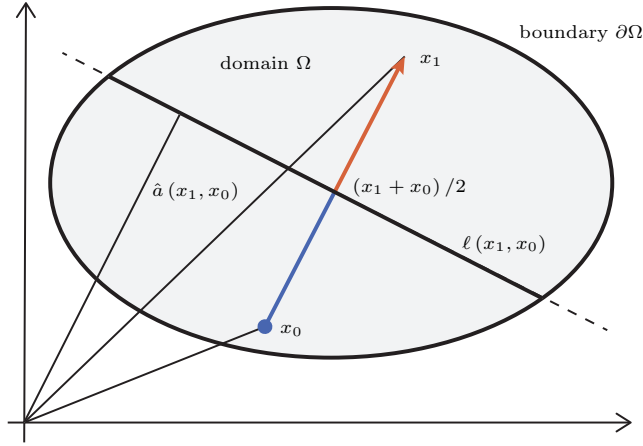}
\caption{Let $\ell\kl{x_1,x_0}$ denote  the line  of all points having
the same distance between $x_1$ and $x_0$.
Then $\hat n\kl{x_1,x_0} =  \skl{x_1 - x_0}/ \sabs{x_1-x_0}$ is the normal vector to
$\ell\kl{x_1,x_0}$   and $\hat a \kl{x_1,x_0} = \skl{\sabs{x_1}^2 - \sabs{x_0}^2}/ \kl{2 \sabs{x_1 -x_0}}$
is the oriented distance of the line $\ell\kl{x_1,x_0}$ from the origin.\label{fig:geometry}
}
\end{figure}
\end{psfrags}

\subsection{Main results}
\label{sec:main}

Our first pair of results states that a back-projection type  inversion
formula  applied to the solution  of  the wave equation \req{wave} recovers
any initial data
$f$ modulo the term $\K_\Om f$.

\begin{theorem}[Inversion of the wave equation]\label{thm:wave}
Let $\Om \subset \R^2$ be a convex bounded  domain with smooth boundary $\partial \Om$.
Then,  for  every function  $f \in C_c^\infty \kl{\Om}$ and every reconstruction point  $x_0 \in \Om$, we have
\begin{align} \label{eq:inv-wave-a}
	f\kl{x_0}
	-
	\kl{\K_\Om f} \kl{x_0}
	&=
	\frac{1}{\pi} \,
	\nabla_{x_0} \cdot \int_{\partial \Om}
	\nu_x
	\kl{
	\int_{\abs{x-x_0}}^\infty
	\frac{\kl{\Uo f}\kl{x,t}}
	{ \sqrt{t^2 - \abs{x-x_0}^2} }
	\, \rmd t
	}
	 \rmd s\kl{x} \,,
	  \\
	  f\kl{x_0}
	  -
	 \kl{\K_\Om f} \kl{x_0}
	  &=  \label{eq:inv-wave-b}
	\frac{1}{\pi} \,
	\int_{\partial \Om}
	\nu_x \cdot \kl{x_0-x}
	\kl{
	\int_{\abs{x-x_0}}^\infty
	\frac{ \kl{\partial_{t} t^{-1} \Uo f}
	\kl{x, t} } { \sqrt{t^2-\abs{x-x_0}^2} }
	\, \rmd t
	}
	\rmd s\kl{x}
	\,.
\end{align}
Here $\Uo f$ is the solution of \req{wave}, $\K_\Om f$ is defined by  \req{kern},
$\rmd\ell$ is the usual arc length measure on $\partial \Om$,
$\nu_x$ denotes the outwards pointing unit normal to $\partial \Omega$, and $\nabla_{x_0} \cdot$ denotes the divergence with respect to $x_0$.
\end{theorem}

\begin{proof}
See Section \ref{sec:wave}.
\end{proof}

We also prove the following  corresponding formulas  for recovering a function
from its mean values over circles centered  on  the boundary $\partial\Om$.

\begin{theorem}[Inversion from circular means]\label{thm:means}
Let $\Om \subset \R^2$ be a convex bounded  domain with smooth boundary $\partial \Om$.
Then,  for  every $f \in C_c^\infty \kl{\Om}$ and every  $x_0 \in \Om$, we have
\begin{align} \label{eq:inv-means-a}
	f\kl{x_0}
	 -
	 \kl{\K_\Om  f}\kl{x_0}
	 &=
	 \frac{1}{\pi} \,
	 \nabla_{x_0} \cdot
	 \int_{\partial \Om}
	 \nu_x
	 \kl{
	 \int_{0}^\infty \frac{r
	 \kl{\M f} \kl{x, r}}
	 {r^2 - \abs{x-x_0}^2}
	 \, \rmd r
	 }
	 \rmd s\kl{x}  \,,
	  \\  \label{eq:inv-means-b}
	f\kl{x_0}
	-
	\kl{\K_\Om  f}\kl{x_0}
	 &=
	\frac{1}{\pi} \,
	\int_{\partial \Om} \nu_x \cdot  \kl{x_0-x}
	\kl{
	\int_{0}^\infty
	\frac{ \kl{\partial_r \M f} \kl{x, r }}
	{r^2 - \abs{x-x_0}^2}
	\, \rmd r
	}
	\rmd s\kl{x}
	\,.
\end{align}
Here  both inner integrals have to be  taken in the principal value
sense,  $\M f$ is the circular mean transform defined by \req{means},
and $\K_\Om$,  $\nu_x$,  $\rmd\ell$ and $\nabla_{x_0}$ are as in
Theorem \ref{thm:wave}.
\end{theorem}

\begin{proof}
See Section~\ref{sec:means}.
\end{proof}

\begin{remark}[Smoothing effect of $\K_\Om$] \label{rem:smooth}
If $\Om \subset \R^2$ is smooth and convex, then the Radon transform
of $\chi_\Om$  is smooth except for such pairs of $\kl{n, a} \in S^1 \times \R$
where the line
$\set{x \in \R^2 : n \cdot x =  a}$  is tangential to the boundary
$\partial \Om$.
Notice further, that for $x_1 \neq x_0 \in \Om$ the line  $\ell\kl{x_1,x_0}$
is never tangential to $\partial \Om$; see  Figure~\ref{fig:geometry}. Since the operator $\partial_a^2 \Ho_a$ preserves the locations of the
singularities of $\Ro \chi_\Om$ this shows that the kernel in \req{kern} has at most a weak singularity proportional to $1/\abs{x_1-x_0}$ and hence is
at least smoothing by one degree.
\end{remark}

For special domains, the integral operator $\K_\Om$
may vanish, in which cases Theorems \ref{thm:wave} and
\ref{thm:means} provide exact reconstruction formulas
for the  inversion of the wave equation and the inversion from circular
means, respectively.
We verify that this is indeed the  case for circular and elliptical domains.

\begin{theorem}[Exact inversion for circular and elliptical domains] \label{thm:ell}
Let $\Om \subset\R^2$ be a circular or  elliptical domain and let
$f\colon \R^2 \to \R $ be  a $C^\infty$ function with support in
$\Om$.
Then  $\K_\Om f$ vanishes identically on $\Om$.
In particular,  the following hold:
\begin{enumerate}[itemsep=0em,topsep=0em,label=(\alph*)]
\item \label{it:ell1}
The function $f $  can be recovered  from the solution
of the wave equation \req{wave}  by means of either of the
following formulas:
\begin{align}
\label{eq:ell-wave-a}
f\kl{x_0}
& =
\frac{1}{\pi} \,
\nabla_{x_0} \cdot \int_{\partial \Om} \nu_x
\kl{
\int_{\abs{x-x_0}}^\infty
\frac{ \kl{\Uo f} \kl{x, t} }
{ \sqrt{t^2-\abs{x-x_0}^2} }
\, \rmd t }
\rmd s\kl{x} \,,
\\ \label{eq:ell-wave-b}
f\kl{x_0}
&=
	\frac{1}{\pi} \,
	\int_{\partial \Om}
	\nu_x \cdot \kl{x_0-x}
	\kl{
	\int_{\abs{x-x_0}}^\infty
	\frac{ \kl{\partial_{t} t^{-1} \Uo f}
	\kl{x, t} } { \sqrt{t^2-\abs{x-x_0}^2} }
	\, \rmd t
	}
	\rmd s\kl{x}
	\,.
\end{align}
\item \label{it:ell2}
The function $f$ can be recovered  from the circular averages
$\M f$ defined in \req{means} by means of either of the
following formulas:
\begin{align}\label{eq:ell-means-a}
f\kl{x_0}
& =
\frac{1}{\pi} \,
\nabla_{x_0}
\cdot
\int_{\partial \Om}
\nu_x
\kl{
\int_{0}^\infty \frac{r \kl{\M f}
\kl{x, r}}
{r^2 - \abs{x-x_0}^2}
\, \rmd r
}
\rmd s\kl{x} \,,
 \\ \label{eq:ell-means-b}
f\kl{x_0}
&=
\frac{1}{\pi} \,
	\int_{\partial \Om} \nu_x \cdot
	\kl{x_0-x}
	\kl{
	\int_{0}^\infty
	\frac{ \kl{\partial_r \M f} \kl{x, r }}
	{r^2 - \abs{x-x_0}^2}
	\, \rmd r
	}
	\rmd s\kl{x}
\,.
\end{align}
Here both inner integrals have to be  taken in the principal value
sense.
\end{enumerate}
\end{theorem}

\begin{proof}
See Section \ref{sec:ell}.
\end{proof}

\subsection{Prior work and innovations}

In the case that $\Om$ is  a spherical domain in $d \geq  2$
dimension, various exact back-projection type inversion formulas
 for the inversion of the wave equation and  the inversion from spherical
 means have been derived in \cite {BurBauGruHalPal07,FinPatRak04,FinHalRak07,Kun07a,XuWan05}.
 In particular, the inversion formula~\cite{BurBauGruHalPal07}
 coincides with our  formula~\req{ell-wave-b} for the wave inversion, and
 the formula of \cite{Kun07a} for the inversion from circular means
 (if  rewritten as in \cite{AgrKucKun09};  see \cite{FinRak09} for a different derivation)  coincides with our formula~\req{ell-means-a}.
 However, in \cite{BurBauGruHalPal07,Kun07a} these  formulas
 are neither shown be exact for elliptical  domains, nor are  they
 investigated for more general  domains.

Inversion formulas for the spherical  mean transform on elliptical
domains have been derived recently in~\cite{Pal12}.
The methods and results there are  different from ours and
no statements are made for general convex domains.
Even the reconstruction formula of \cite{Pal12} for ellipses
differs from each of our formulas \req{ell-means-a}, \req{ell-means-b}.
Our  results are more closely related to the one of~\cite{Nat12}, where
corresponding results have derived for convex domains in
$\R^3$.  Note further that  in  three spatial dimensions a statement
similar to our  Theorem~\ref{thm:wave} is also present in~\cite{XuWan05}
(where, however, no explicit expression for the operator $\K_\Om$ has been derived).
Anyway, none of the papers \cite{Nat12,XuWan05} considers
the two-dimensional case.

\subsection{Outline}

The following sections are devoted to the proofs of the
theorems presented  in  Section \ref{sec:main}. In Section  \ref{sec:aux}
we derive auxiliary statements that we require for the proofs  of these
results.
In Section~\ref{sec:wave} we then derive the
formulas for the wave inversion claimed in Theorem~\ref{thm:wave}
and in Section~\ref{sec:means} we  establish the corresponding results for the inversion from circular
means.
The cases of circular and elliptical
domains  are considered in Section~\ref{sec:ell},
where we show that the operator   $\K_\Om$ vanishes in these
cases and hence we establish the exact inversion
formulas claimed in Theorem~\ref{thm:ell}.
The paper concludes with  a  short discussion in Section~\ref{sec:conclusion}.

\section{Auxiliary results}
\label{sec:aux}

For the following considerations recall the definitions of the circular mean transform $\M$, the Radon transform $\Ro$, and the operator $\Uo$ that maps
any initial data $f\in C_c^\infty \kl{\Om}$ to the solution of \req{wave}; see Section~\ref{sec:not}.
We will further make use of the operator  $\Vo$,
which maps any $f \in C_c^\infty \kl{\Om}$ to the solution
$v \colon \R^2\times \kl{0, \infty} \to \R $  of the initial value problem
\begin{equation}  \label{eq:wave2}
	\left\{
	 \begin{aligned}
	\kl{\partial_{t} ^2  - \Delta_{x} }
	v\kl{x, t}
	&=
	0
	& \text{ for }
	\kl{x,t} \in
	\R^2 \times \kl{0, \infty}
	\\
	v\kl{x,0}
	&=
	0
	& \text{ for }
	x  \in \R^2
	\\
	\kl{\partial_{t} v}
	\kl{x,0}
	&
	= f \kl{x}
	&
	\text{ for }
	x  \in \R^2
\end{aligned} \;.
\right.
\end{equation}
One easily recognizes that $\Uo$ and $\Vo$ are related by
the identity $\Uo f =  \partial_{t} \Vo f$. Indeed, if $v$  solves \req{wave2}, then by definition  $u = \partial_{t}v$  satisfies
the wave equation $\skl{\partial_{t} ^2  - \Delta_{x} } u  = 0 $ and the initial condition $u\kl{x,0} = f\kl{0}$. Further, the wave equation yields $\skl{\partial_{t}u}\kl{x,0} = \skl{\partial_{t}^2 v}\kl{x,0}=  \kl{\Delta_{x}v}\kl{x,0} = 0$ which shows that $u$
is a solution of  \req{wave}.

Next we proof  a simple integral identity that we require
for the following considerations.

\begin{lemma}\label{lem:intphi}
Suppose that $f \in C_c^\infty\kl{\Om}$ and that  $\phi \colon \R^2 \to \R$ is continuously differentiable on $\bar \Omega$ and
vanishes outside $\bar \Omega$.
Then, for every $x_1 \in \Om$, we have
\begin{multline}\label{eq:intphi}
\int_{\R^2}
\phi\kl{x}
\kl{\int_{0}^\infty
\kl{\partial_{r_0} \M f}\kl{x, r_0}
\ln \abs{r_0^2 - \abs{x-x_1}^2}
\rmd r_0}
\rmd x
\\=
\frac{1}{2\pi}
\int_{\R^2}
\frac{f\kl{x_0}}{\abs{x_1-x_0}}
\kl{\int_{\R}
\kl{\partial_a\Ro\phi}\kl{\hat n, a}
\ln \abs{2 \abs{x_1-x_0} \kl{a-\hat a}}
\rmd a}
\rmd x_0 \,.
\end{multline}
Here  $\hat n = \hat n\kl{x_1, x_0}$ and
$\hat a = \hat a \kl{x_1, x_0}$ are defined by
Equation~\req{hatns}.
\end{lemma}

\begin{proof}
We first verify \req{intphi} with  $\Phi\kl{\,\cdot\,}$ in place of $\ln \abs{\,\cdot\,}$, where $\Phi \colon \R \to \R$ is  continuously differentiable and integrable. Hence we show
that  for every $x_1 \in \Om$ the following holds:
\begin{multline}\label{eq:intphi2}
\int_{\R^2}
\phi\kl{x}
\kl{\int_{0}^\infty
\kl{\partial_{r_0} \M f}\kl{x, r_0}
\Phi \kl{r_0^2 - \abs{x-x_1}^2}
\rmd r_0}
\rmd x
\\=
\frac{1}{2\pi}
\int_{\R^2}
\frac{f\kl{x_0}}{\abs{x_1-x_0}}
\kl{\int_{\R}
\kl{\partial_a\Ro\phi}\kl{\hat n, a}
\Phi \kl{2 \abs{x_1-x_0} \kl{a-\hat a}}
\rmd a}
\rmd x_0 \,.
\end{multline}
After performing one integration by parts, using the definition of the circular mean transform, and introducing polar coordinates around the center $x$ afterwards, the inner integral on the left hand side  of  \req{intphi2} can be written as
\begin{align*}
\int_{0}^\infty
\kl{\partial_{r_0} \M f}\kl{x, r_0}
\Phi\kl{r_0^2 - \abs{x-x_1}^2}
\rmd r_0
&=
-2
\int_{0}^\infty
r_0
\kl{\M f}\kl{x, r_0}
\Phi'\kl{r_0^2 - \abs{x-x_1}^2}
\rmd r_0
\\
&=
-\frac{1}{\pi}
\int_{0}^\infty
\int_{S^1}
r_0
f\kl{x + r_0\om_0}
\Phi'\kl{r_0^2 - \abs{x-x_1}^2}
\rmd\om_0
\rmd r_0
\\
&=
-\frac{1}{\pi}
\int_{\Om}
f\kl{x_0}
\Phi'\kl{\abs{x-x_0}^2 - \abs{x-x_1}^2}
\rmd x_0 \,.
\end{align*}
By straight forward computation one verifies that  $\abs{x-x_0}^2 - \abs{x-x_1}^2
= 2 \kl{x_1-x_0} \cdot \kl{x - \kl{x_1+x_0}/2 }$.
Therefore,  one application of Fubini's
theorem, the definition of the Radon transform and the definitions of
$\hat n = \hat n\kl{x_1, x_0}$ and
$\hat a = \hat a \kl{x_1, x_0}$  yield
\begin{multline*}
\int_{\R^2}
\phi\kl{x}
\int_{0}^\infty
\kl{\partial_{r_0} \M f}\kl{x, r_0}
\Phi\kl{r_0^2 - \abs{x-x_1}^2}
\rmd r_0
\rmd x
\\
\begin{aligned}
&=
-
\frac{1}{\pi}
\int_{\Om}  f\kl{x_0}
\int_{\R^2}
\phi\kl{x}
\Phi'\kl{2 \kl{x_1-x_0} \cdot \kl{ x - \kl{x_1-x_0}/2}}
\rmd x
\rmd x_0
\\
&=
-\frac{1}{\pi}\int_{\Om}
f\kl{x_0}
\kl{
\int_{\R}
\int_{\R}
\phi\kl{a \hat n + b \hat n^\bot}
\Phi'\kl{2 \abs{x_1-x_0} \kl{ a - \hat a}}
\rmd b
\rmd a}
\rmd x_0
\\
&=
-\frac{1}{\pi}\int_{\Om}
f\kl{x_0}
\kl{\int_{\R}
\kl{\Ro\phi}\kl{\hat n, a}
\Phi'\kl{2 \abs{x_1-x_0} \kl{a - \hat a}}
\rmd a}
\rmd x_0
\end{aligned}
\end{multline*}
Finally, after performing one further integration by parts in the inner integral, the last expression is seen to be equal to
\begin{equation*}
\frac{1}{2\pi}
\int_{\Om}
\frac{f\kl{x_0}}{\abs{x_1-x_0}}
\kl{\int_{\R}
\kl{\partial_a\Ro\phi}\kl{\hat n, a}
\Phi\kl{2 \abs{x_1-x_0} \kl{a-\hat a}}
\rmd a}
\rmd x_0 \,.
\end{equation*}
This shows  that \req{intphi2} indeed
holds for all  $\Phi \in C^1\kl{\R} \cap L^1\kl{\R}$.
In order to verify the corresponding identity for $\ln \abs{\,\cdot\,}$ in place of $\Phi$
we may write $\ln \abs{\,\cdot\,}$ as the $L^1$-limit of continuously differentiable functions $\Phi_n$.
Application of \req{intphi2} with $\Phi_n$ in place of $\Phi$  and taking the limit $n\to \infty$ afterwards then implies
the desired identity~\req{intphi}.
\end{proof}

The next  Lemma is the key for the results in this paper.

\begin{lemma}\label{lem:aux-main}
For all $f, g \in C_c^\infty \kl{\Om}$ we have
\begin{multline}\label{eq:aux-main}
\int_\Om
\kl{
\int_0^\infty
\kl{\Uo f}  \kl{x, t}
\kl{\Vo g}  \kl{x, t}
\,
\rmd t}
\rmd x
 \\
=
- \frac{1}{8\pi}
\int_\Om
\int_\Om
f\kl{x_0}
g\kl{x_1}
\frac{ \kl{\Ho_a \Ro \chi_\Om}
\kl{\hat n\kl{x_1,x_0},\hat a\kl{x_1,x_0}}}
{\abs{x_1-x_0}}
\, \rmd x_0
 \rmd x_1
 \,.
\end{multline}
Here $\Ho_a$ is the Hilbert transform in the second argument,
$\chi_\Om$ is the characteristic function of $\Om$,
and $\hat n = \hat n\kl{x_1, x_0}$ and $\hat a = \hat a \kl{x_1, x_0}$ are defined by
Equation~\req{hatns}.
\end{lemma}

\begin{proof}
First recall  that the  explicit expressions for the solutions of
the initial value problems  \req{wave2} and  \req{wave}
(see, for example, \cite[p. 134]{Joh82})
are given by
\begin{align*} 
 \kl{\Vo g} \kl{x,t}
 &=
 \int_{0}^{t}
 \frac{r \kl{\M g} \kl{x, r } }
 {\sqrt{t^2 - r^2}} \, \rmd r \,,
\\ 
 \kl{\Uo f} \kl{x,t}
 &=
 \partial_{t}
 \int_{0}^{t}
 \frac{r \kl{\M f} \kl{x, r } }
 {\sqrt{t^2 - r^2}} \, \rmd r
 \,.
 \end{align*}
After performing  one integration by parts and differentiating under the integral sign, the expression for  $\Uo f$ can be rewritten as
\begin{equation}\label{eq:sol-UU2}
 \kl{ \Uo  f}\kl{x,t}
 =
\partial_{t}
\int_{0}^{t }
 \kl{ \partial_r \M f} \kl{x, r }
\sqrt{t^2 - r^2}
\,  \rmd r
\\ =
\int_{0}^{t }
\kl{ \partial_r \M f} \kl{x, r }
\frac{t}{\sqrt{t^2 - r^2}}
\, \rmd r
 \,.
\end{equation}
These identities and Fubini's theorem yield
\begin{multline*}
\int_{0}^\infty
\kl{\Uo f} \kl{x,t}
\kl{\Vo g} \kl{x,t}
\rmd t
\\
\begin{aligned}
&=
\lim_{T \to \infty}
\int_0^T
t
\kl{\int_0^t
\frac{\kl{\partial_{r_0}\M f}\kl{x,r_0}}{\sqrt{t^2 - r_0^2}}
\rmd r_0}
\kl{\int_0^t
\frac{\kl{r_1 \M g}\kl{x,r_1}}{\sqrt{t^2 - r_1^2}}
\rmd r_1} \rmd t
\\&
=
\lim_{T \to \infty}
\int_0^\infty
\kl{r_1 \M g}\kl{x,r_1}
\int_0^\infty
\kl{\partial_{r_0}\M f}\kl{x,r_0}
\kl{\int_{\max \set{r_0, r_1}}^T
\frac{t \rmd t
}{\sqrt{t^2 - r_0^2}\sqrt{t^2 - r_1^2}}
}
\rmd r_0
\rmd r_1
\,.
\end{aligned}
\end{multline*}
The  inner integral in the above expression
evaluates to
\begin{equation} \label{eq:intgreen}
\int_{\max\set{r_0, r_1}}^T
 \frac{t \, \rmd t}{\sqrt{t^2 - r_1^2}\sqrt{t^2 - r_0^2}}
\\
=
\ln \kl{\sqrt{T^2 - r_0^2} +  \sqrt{T^2 - r_1^2}}
- \frac{1}{2}  \ln  \kl{ \abs{r_0^2 - r_1^2 } }  \,.
\end{equation}
Hence we obtain
\begin{multline*}
\int_{0}^\infty
\kl{\Uo f} \kl{x,t}
\kl{\Vo g} \kl{x,t}
\rmd t
\\
\begin{aligned}
&=
\lim_{T \to \infty}
\int_0^\infty
\int_0^\infty
\kl{\partial_{r_0}\M f}\kl{x,r_0}
\kl{r_1\M g}\kl{x,r_1}
\ln \kl{ \sqrt{T^2 - r_0^2} + \sqrt{T^2 - r_1^2}  }
\rmd r_0
\rmd r_1
\\
&\hspace{1cm} -\frac{1}{2}
\int_0^\infty
\int_0^\infty
\kl{\partial_{r_0}\M f}\kl{x,r_0}
\kl{r_1\M g}\kl{x,r_1}
 \ln   \abs{ r_0^2 - r_1^2}
\rmd r_0
\rmd r_1
\,.
\end{aligned}
\end{multline*}

After one integration by  parts and performing the limit $T\to \infty$ afterwards,
the first term on the  right hand side is seen to vanish.
Therefore, by integrating over the variable $x$,
interchanging the order of integration and introducing polar
coordinates, we obtain
\begin{multline*}
\int_\Om
\int_0^\infty
\kl{\Uo f} \kl{x,t}
\kl{\Vo g} \kl{x,t}
\rmd t
\rmd x
\\
=
-\frac{1}{4\pi}
\int_\Om
g\kl{x_1}
\kl{\int_{\R^2}\chi_\Om \kl{x}
\int_0^\infty
\kl{\partial_{r_0}\M f}\kl{x,r_0}
 \ln   \abs{r_0^2 - \abs{x-x_1}^2}
\rmd r_0
\rmd x }
\rmd x_1
\,.
\end{multline*}
According to Lemma \ref{lem:intphi}, the double integral in brackets
equals
\begin{equation*}
\frac{1}{2\pi}
\int_{\Om}
\frac{f\kl{x_0}}{\abs{x_1-x_0}}
\kl{
\int_{\R}
\kl{\partial_a\Ro\chi_\Om}\kl{\hat n, a}
\ln \abs{2 \abs{x_1-x_0} \kl{a - \hat a}}
\rmd a
}
\rmd x_0\,.
\end{equation*}
Next note that the distributional derivative of  $\ln \abs{a}$ is $\mathrm{P.V.} \frac{1}{a}$.
After recalling the definition of the Hilbert transform we therefore obtain 
\begin{multline*}
\int_\Om
\int_0^\infty
\kl{\Uo f} \kl{x,t}
\kl{\Vo g} \kl{x,t}
\rmd t
\rmd x
\\
\begin{aligned}
&=
-\frac{1}{8\pi^2}
\int_\Om
\int_{\Om}
\frac{f\kl{x_0}g\kl{x_1}}{\abs{x_1-x_0}}
\kl{\int_{\R}
\kl{\partial_a\Ro\chi_\Om}\kl{\hat n, a}
\ln \abs{2 \abs{x_1-x_0} \kl{a-\hat a}}
\rmd a}
 \rmd x_0
\, \rmd x_1
\\
&=
-
\frac{1}{8\pi^2}
\int_\Om
\int_\Om
\frac{f\kl{x_0}g\kl{x_1}}{\abs{x_1-x_0}}
\kl{\mathrm{P.V.} \int_{\R}
\kl{\Ro\chi_\Om}\kl{\hat n, a}
\frac{\rmd a}{\hat a - a} }
\rmd x_0
\, \rmd x_1
\\
&=
-
\frac{1}{8\pi}
\int_\Om
\int_\Om
\frac{f\kl{x_0}g\kl{x_1}}{\abs{x_1-x_0}}
\kl{\Ho_a \Ro\chi_\Om}\kl{\hat n, \hat a}
\,
\rmd x_0
\rmd x_1
\,.
\end{aligned}
\end{multline*}
Hence we have verified~\req{aux-main} which concludes the proof
\end{proof}

We are now ready to derive the following auxiliary theorem  from which we will
extract all formulas presented in the introduction.

\begin{theorem} \label{thm:main-aux}
For all $f, g \in C_c^\infty \kl{\Om}$  we have
\begin{multline}\label{eq:green}
 -2
 \int_{\partial\Om}
 \nu_x \cdot \kl{
\int_0^\infty
 \kl{\Uo f} \kl{x, t}
\kl{\nabla_{x} \Vo g} \kl{x, t}
\rmd t }
\rmd s\kl{x}
\\
=
\int_\Om f \kl{x_0}
g\kl{x_0}  \rmd x_0
-
\int_\Om
\kl{\K_\Om f}
\kl{x_0}
g \kl{x_0}  \rmd x_0
\,.
\end{multline}
Here $\Uo$ and $\Vo$ denote the solution operators of the initial value problems
\req{wave} and \req{wave2}, respectively, $\K_\Om f$  is defined by \req{kern},
and $\nu_x$ denotes the outward pointing unit
normal to $\partial \Om$.
\end{theorem}

\begin{proof}
Application of  Greens second identity to the functions
$ \kl{\Uo f} \kl{\edot, t}$ and   $ \kl{\Vo g} \kl{\edot, t}$ with fixed
$t$ followed by an integration  over the temporal variable yields
\begin{multline*}
\int_0^{\infty}
\kl{ \int_\Om
\kl{\kl{\Vo g}\kl{x,t} \kl{\Delta_{x}\Uo f} \kl{x,t}
-
\kl{\Uo f}\kl{x,t} \kl{\Delta_{x} \Vo g}  \kl{x,t} } \rmd x 
} \rmd t
\\=
\int_0^{\infty}
\kl{
\int_{\partial\Om} \nu_x
\kl{ \kl{\Vo g}\kl{x,t} \kl{\nabla_{x}\Uo f} \kl{x,t}
- \kl{\Uo f}\kl{x,t} \kl{\nabla_{x} \Vo g}  \kl{x,t} }
\rmd s\kl{x}
} \rmd t\,.
\end{multline*}
Since $\Vo f$ and $\Uo g$ are both solutions  of the  wave  equation
we have the equalities
$\Delta_{x}\Uo f   = \partial_{t}^2\Uo f $
and
$\Delta_{x} \Vo g  = \partial_{t}^2 \Vo g $.
Together with the relation
$ \kl{\Vo g} \kl{\nabla_{x}\Uo f} = \nabla_{x}\kl{\Vo g\Uo f} -
 \kl{\nabla_{x} \Vo g}\kl{\Uo f}$ and one application of the divergence theorem this further implies
 \begin{multline*}
\int_\Om \kl{
\int_0^{\infty}
\kl{\kl{\Vo g}\kl{x,t} \kl{\partial_{t}^2 \Uo f} \kl{x,t}
-
\kl{\Uo f}\kl{x,t} \kl{\partial_{t}^2 \Vo g}  \kl{x,t} } \rmd t} \rmd x
\\=
\int_0^{\infty} \kl{
\int_{\Om}
\Delta_{x} \kl{ \kl{\Vo g} \kl{x,t} \kl{\Uo f} \kl{x,t} }  \rmd x } \rmd t
- 2
\int_0^{\infty} \kl{
\int_{\partial \Om} \nu_x
\kl{\kl{\Uo f}\kl{x,t} \kl{\nabla_{x} \Vo g}  \kl{x,t} } \rmd s\kl{x}  }
\rmd t  \,.
\end{multline*}
After integrating the inner integrals on the right hand side by parts  and
using the initial conditions $\kl{\Uo f} \kl{\edot, 0}  = f$, $\kl{\partial_{t}\Uo f} \kl{\edot, 0}  = 0$ and $\kl{\partial_{t}\Vo g} \kl{\edot, 0}  = g$,
the expression on the left hand side evaluates to
$\int_\Om f \kl{x_0} g\kl{x_0} \rmd x_0$
which yields
\begin{multline}\label{eq:green-t}
\int_\Om f\kl{x_0}
g\kl{x_0}  \rmd x_0
=
\int_\Om
\int_0^\infty
\Delta_{x} \kl{ \kl{\Uo f}\kl{x, t}
\kl{\Vo g} \kl{x, t} }
\rmd t
\rmd x
 \\
 -2
 \int_{\partial\Om}
  \nu_x
  \int_0^\infty
 \kl{\Uo f} \kl{x, t}
\kl{\nabla_{x} \Vo g} \kl{x, t}
\rmd t
\rmd s\kl{x}
\,.
\end{multline}

Next notice that
\begin{align*}
\Delta_{x}
\kl{\kl{\Uo f}
\kl{\Vo g}
}
& =
\kl{\Delta_{x} \Uo f}
\kl{\Vo g}
+
2 \kl{ \nabla_{x} \Vo f}
\cdot \kl{\nabla_{x} \Uo g}
+
 \kl{\Uo f}
\kl{\Delta_{x} \Vo g}
\\
& =
\kl{ \Uo \Delta_{x} f}
\kl{\Vo g}
+
2 \kl{ \Vo \nabla_{x}  f}
\cdot \kl{ \Uo \nabla_{x} g}
+
 \kl{\Uo f}
\kl{\Vo \Delta_{x}  g} \,,
\end{align*}
where $\Vo \nabla_{x}  f$ (and similarly $\Uo \nabla_{x}  g$) is  defined
by component-wise application of $\Vo$  to $\nabla_{x} f$.
Hence in view of Lemma~\ref{lem:aux-main}
and by repeated application of the divergence theorem we
obtain
\begin{multline*}
\int_\Om
\int_0^\infty
\Delta_{x}
\kl{\kl{\Uo f}  \kl{x, t}
\kl{\Vo g}  \kl{x, t}
}
\rmd t
\rmd x
 \\
 \begin{aligned}
=& - \frac{1}{8\pi}
\int_\Om
\int_\Om
\kl{\nabla_{x_0}+\nabla_{x_1}}^2
\kl{ f\kl{x_0}
g\kl{x_1}
}
\frac{ \kl{\Ho_a \Ro \chi_\Om}
\kl{\hat n\kl{x_1,x_0},\hat a\kl{x_1,x_0}}}
{\abs{x_1-x_0}}
 \, \rmd x_0
 \rmd x_1
\\
 =&
 - \frac{1}{8\pi}
\int_\Om
\int_\Om
f\kl{x_0}
g\kl{x_1}
\kl{\nabla_{x_0}+\nabla_{x_1}}^2
\kl{
\frac{ \kl{\Ho_a \Ro \chi_\Om}
\kl{\hat n\kl{x_1,x_0},\hat a\kl{x_1,x_0}}}
{\abs{x_1-x_0}}
}
 \, \rmd x_0
 \rmd x_1
 \,.
\end{aligned}
\end{multline*}

As easily verified, the following relations hold true:
\begin{align*}
&\kl{\nabla_{x_1}  + \nabla_{x_0}} \hat n \kl{x_1,x_0}
 =0
\,, \\
&\kl{\nabla_{x_1} + \nabla_{x_0}} \hat a \kl{x_1,x_0}
 =
\kl{ x_1 - x_0}/ \abs{x_1- x_0}
\,, \\
&\kl{\nabla_{x_1}  + \nabla_{x_0}} \abs{x_1-x_0}^{-1}
=0
\,, \\
&\kl{\nabla_{x_1}  + \nabla_{x_0}} \cdot \kl{x_1-x_0}
=0
\,. \\
\end{align*}
These relations imply
\begin{equation*}
\kl{\nabla_{x_1} + \nabla_{x_0}}^2
\frac{\kl{ \Ho_a\Ro \chi_\Om }
\kl{\hat n, \hat a} }{8\pi\abs{x_1  -  x_0} }
=
\kl{\nabla_{x_1} + \nabla_{x_0}}
\frac{x_1  -  x_0}{8\pi \abs{x_1  -  x_0}^2}
\kl{\partial_a\Ho_a \Ro \chi_\Om }
\kl{\hat n, \hat a}
=
\frac{\kl{\partial_a^2\Ho_a \Ro \chi_\Om }
\kl{\hat n, \hat a}
}{8\pi\abs{x_1  -  x_0}}
\,,
\end{equation*}
and therefore
\begin{multline*}
\int_\Om
\int_0^\infty
\Delta_{x}
\kl{\kl{\Uo f}  \kl{x, t}
\kl{\Vo g}  \kl{x, t}
}
\rmd t
\rmd x
\\
=
 - \frac{1}{8\pi}
\int_\Om
\int_\Om
f\kl{x_0}
g\kl{x_1}
\frac{ \kl{\partial_a^2\Ho_a \Ro \chi_\Om}
\kl{\hat n\kl{x_1,x_0},\hat a\kl{x_1,x_0}}}
{\abs{x_1-x_0}}
 \,
 \rmd x_0
 \rmd x_1
=
\int_\Om
\kl{\K_\Om f }
\kl{x_1}
g\kl{x_1}
 \rmd x_1
 \,.
\end{multline*}
Inserting the last displayed equation into Equation~\req{green-t}
finally yields the claimed identity~\req{green}.
\end{proof}

\section{Proof of Theorem~\ref{thm:wave}}
\label{sec:wave}

In this section we establish the reconstruction formulas stated in
Theorem~\ref{thm:wave} as a consequence of
Theorem~\ref{thm:main-aux} and the following Lemma that is easy
to establish.

\begin{lemma} \label{lem:aux-wave}
For every  $f, g \in C_c^\infty \kl{\Om}$  we have
\begin{multline*}
 \int_{\partial\Om}
 \nu_x \cdot
\int_0^\infty
 \kl{\Uo f} \kl{x, t}
\kl{\nabla_{x} \Vo g} \kl{x, t}
\rmd t
\rmd s\kl{x}
\\
=
-
\frac{1}{2\pi}
\int_{\R^2}
\kl{\nabla_{x_0}
\cdot
\int_{\partial\Om}
\nu_x
\int_{\abs{x-x_0}}^\infty
  \frac{\kl{\Uo f} \kl{x, t}}{\sqrt{t^2 -\abs{x-x_0}^2}}
\,
\rmd t
\rmd s\kl{x}
}g\kl{x_0}
\rmd x_0\,.
\end{multline*}
\end{lemma}

\begin{proof}
The identity  $\nabla_{x} \Vo g =  \Vo\nabla_{x} g$  (again $\Vo\nabla_{x} g$ is defined component-wise)
and the explicit expression for the solution of the wave equation \req{wave2} imply
\begin{multline*}
 \int_{\partial\Om}
 \nu_x \cdot
\int_0^\infty
 \kl{\Uo f} \kl{x, t}
\kl{\nabla_{x} \Vo g} \kl{x, t}
\rmd t
\rmd s\kl{x}
\\
\begin{aligned}
&=
\phantom{-}
\frac{1}{2\pi}
\int_{\partial\Om}
\nu_x \cdot
\int_0^\infty
\kl{\Uo f} \kl{x, t}
\int_{\R^2} \kl{\nabla_{x_0} g}\kl{x_0} \frac{\chi\set{\abs{x-x_0}< t}}{\sqrt{t^2 -\abs{x-x_0}^2}}
\rmd x_0
\rmd t
\rmd s\kl{x}
\\
&=
\phantom{-}
\frac{1}{2\pi}
\int_{\R^2}
\kl{\nabla_{x_0} g}\kl{x_0}
\cdot
\int_{\partial\Om}
\nu_x
\int_{\abs{x-x_0}}^\infty
  \frac{\kl{\Uo f} \kl{x, t}}{\sqrt{t^2 -\abs{x-x_0}^2}}
\rmd t
\rmd s\kl{x}
\rmd x_0
\\
&=
-
\frac{1}{2\pi}
\int_{\R^2}
\kl{\nabla_{x_0}
\cdot
\int_{\partial\Om}
\nu_x
\int_{\abs{x-x_0}}^\infty
  \frac{\kl{\Uo f} \kl{x, t}}{\sqrt{t^2 -\abs{x-x_0}^2}}
\,
\rmd t
\rmd s\kl{x}
}
g\kl{x_0}
\rmd x_0 \,.
\end{aligned}
\end{multline*}
This is already the desired identity and concludes the proof of Lemma \ref{lem:aux-wave}.
\end{proof}

\subsection{Proof of formula \req{inv-wave-a}}

By Theorem~\ref{thm:main-aux} and Lemma~\ref{lem:aux-wave}, for all $f$, $g \in C_c^\infty\kl{\Om}$ the following identity
holds true:
\begin{multline*}
\frac{1}{\pi}
\int_{\R^2}
\kl{\nabla_{x_0}
\cdot
\int_{\partial\Om}
\nu_x
\int_{\abs{x-x_0}}^\infty
  \frac{\kl{\Uo f} \kl{x, t}}{\sqrt{t^2 -\abs{x-x_0}^2}}
\,
\rmd t
\rmd s\kl{x} }
g\kl{x_0}
\rmd x_0
\\
=
\int_{\R^2} f \kl{x_0}
g\kl{x_0}  \rmd x_0
-
\int_{\R^2}
\kl{\K_\Om f}
\kl{x_0}
g \kl{x_0}  \rmd x_0 \,.
\end{multline*}
This implies  that~\req{inv-wave-a} holds almost everywhere. Since
both sides of  \req{inv-wave-a} are continuous functions
the equality   must also hold pointwise.

\subsection{Proof of formula~\req{inv-wave-b}}

The second  identity given in Theorem \ref{thm:wave} is an
easy corollary of the first one just established. Indeed, from
the chain rule and one integration by parts we obtain
\begin{multline*}
	\frac{1}{\pi}
\nabla_{x_0} \cdot \int_{\partial \Om}
	\nu_x
	\kl{
	\int_{\abs{x-x_0}}^\infty
	\frac{\kl{\Uo f}\kl{x,t}}
	{ \sqrt{t^2 - \abs{x-x_0}^2} }
	  \, \rmd t
	  }
	  \rmd s\kl{x}
	  \\
	  \begin{aligned}
	&=
	\phantom{-}	\frac{1}{\pi} 
\nabla_{x_0} \cdot \int_{\partial \Om}
	\nu_x
	\int_{\abs{x-x_0}}^\infty
	\partial_t
	\kl{ \sqrt{t^2 - \abs{x-x_0}^2} }
	 \kl{t^{-1} \Uo f }\kl{x,t} \, \rmd t \, \rmd s\kl{x}
	\\
	&=
	-\frac{1}{\pi}
\nabla_{x_0} \cdot \int_{\partial \Om}
	\nu_x
	\int_{\abs{x-x_0}}^\infty
	\sqrt{t^2 - \abs{x-x_0}^2}
	 \kl{\partial_t t^{-1} \Uo f }\kl{x,t} \, \rmd t \, \rmd s\kl{x}
	\\&=
	\phantom{-}	\frac{1}{\pi} 
\int_{\partial \Om}
	\nu_x\cdot
	\kl{x_0-x}
	\kl{
	\int_{\abs{x-x_0}}^\infty
	\frac{ \kl{\partial_{t} t^{-1} \Uo f}
	\kl{x,t} } { \sqrt{t^2-\abs{x-x_0}^2} }
	 \, \rmd t
	 }
	 \rmd s\kl{x} 	\,.
	\end{aligned}
\end{multline*}
In view of~\req{inv-wave-a},  this verifies
formula~\req{inv-wave-b}.

\section{Proof of  Theorem~\ref{thm:means}}
\label{sec:means}

In this section we derive the inversion formulas for the circular
mean transform stated in Theorem~\ref{thm:means}.
The proofs will be based on the inversion formula
\req{inv-wave-a} for the inversion of the wave equation and the
explicit expression \req{sol-UU2} for  the solution of~\req{wave}
in terms of  the circular means  $\M f$.

\subsection{Proof of  formula \req{inv-means-a}}

The inversion formula  \req{inv-wave-a} for the wave equation
and the explicit expression \req{sol-UU2} for  the solution of~\req{wave}
imply that
\begin{multline} \label{eq:mw}
f\kl{x_0}
-
\kl{\K_\Om f}\kl{x_0}
=
\frac{1}{\pi} \,
\nabla_{x_0} \cdot
\int_{\partial \Om} \nu_x
\kl{
\int_{\abs{x-x_0}}^\infty \frac{ \Uo f \kl{x, t}}{\sqrt{t^2 -\abs{x-x_0}^2 }}
\, \rmd t
}
\rmd s\kl{x}
\\
=
\frac{1}{\pi} \,
\nabla_{x_0} \cdot
\int_{\partial \Om} \nu_x
\kl{
 \lim_{T \to \infty}
\int_{\abs{x-x_0}}^T
\int_{0}^{t}
\frac{
t\,
 \kl{ \partial_r \M f} \kl{x, r } }
{\sqrt{t^2 - r^2}
\sqrt{t^2 -\abs{x-x_0}^2}
}
\, \rmd r
\rmd t }
 \,\rmd s\kl{x}
\,.
\end{multline}

After changing the order of integration and using Equation \req{intgreen},  the  inner double integral
evaluates  to
\begin{multline*}
\lim_{T \to \infty}
\int_{0}^\infty
\partial_r \kl{ \M f} \kl{x, r }
\int_{\max\set{\sabs{x-x_0}, r}}^T
\frac{t \, \rmd t }
{\sqrt{t^2 - r^2}
\sqrt{t^2 -\abs{x-x_0}^2}
}
\,
\rmd r
\\
=
- \frac{1}{2}
\int_{0}^\infty
 \kl{ \partial_r \M f} \kl{x, r }
 \ln \abs{r^2 - \sabs{x-x_0}^2}
\rmd r
 \,.
\end{multline*}
Together with  Equation~\req{mw} and
one integration by parts  (using that the  distributional derivative of $\ln \abs{r}$ is $1/r$)  this yields
\begin{multline}
f\kl{x_0}
-
\kl{\K_\Om f}\kl{x_0}
=
-\frac{1}{2\pi} \,
\nabla_{x_0} \cdot
\int_{\partial \Om} \nu_x
\int_{0}^\infty
 \kl{\partial_r \M f} \kl{x, r }
 \ln \abs{r^2 - \sabs{x-x_0}^2}
 \,  \rmd r
\, \rmd s\kl{x}
\\
=
\frac{1}{\pi}  \;
\nabla_{x_0} \cdot
\int_{\partial \Om} \nu_x
\kl{
\int_{0}^\infty
 \frac{ \kl{r \M f} \kl{x, r }
}{r^2 - \abs{x-x_0}^2}
 \,  \rmd r
}
\rmd s\kl{x} \,,
\end{multline}
where  the  last integral is taken in the principal value
sense. The  last identity obviously coincides with the inversion formula
\req{inv-means-a}.

\subsection{Proof of  formula~\req{inv-means-b}}

The  second  inversion formula  \req{inv-means-b} for the
circular mean transform $\M$   could be obtained  from
\req{inv-wave-b} in way similar (but slightly more involved)
to the derivation of \req{inv-means-a} presented above.
However, it is simpler  to derive formula \req{inv-means-b} directly
 from \req{inv-means-a}.

One integration by  parts  and interchanging the order of differentiating and integration  yields
\begin{multline*}
	\frac{1}{\pi}
	\,
	\nabla_{x_0} \cdot
	\kl{
	\int_{\partial \Om}
	\nu_x
	\int_0^\infty
	\frac{r \M f\kl{x,r} }
	{ r^2 - \abs{x-x_0}^2 }
	 \, \rmd r
	 }
	 \rmd s\kl{x}
	  \\
	  \begin{aligned}
	&=
	-\frac{1}{2\pi} \,
\nabla_{x_0} \cdot
\int_{\partial \Om} \nu_x
\int_{0}^\infty
 \kl{\partial_r \M f} \kl{x, r }
 \ln \abs{r^2 - \sabs{x-x_0}^2}
 \, \rmd r
 \, \rmd s\kl{x}
\\
&=
	-
	\frac{1}{2\pi}
	\,
	\int_{\partial \Om}
	\nu_x \cdot
	 \int_{0}^\infty
	\nabla_{x_0} \kl{  \ln \abs{r^2 - \sabs{x-x_0}^2} }
	\kl{ \partial_r  \M f}
	\kl{x, r}
	\, \rmd r
	\, \rmd s\kl{x}
	\\&=
	\phantom{-}
	\frac{1}{\pi}
	\,\int_{\partial \Om}
	\nu_x\cdot
	\kl{x_0-x}
	\kl{
	\int_{0}^\infty
	\frac{ \kl{ \partial_r \M f}
	\kl{x, r} }
	{ r^2-\abs{x-x_0}^2 }
	\, \rmd r
	}
	\rmd s\kl{x}\,.
	\end{aligned}
\end{multline*}
In view of  \req{inv-means-a} this
verifies the formula  \req{inv-means-b}.

\section{Proof of Theorem~\ref{thm:ell}}
\label{sec:ell}

In this section we verify the exact reconstruction formulas
presented in  Theorem~\ref{thm:ell} in the case that the domain $\Omega$ is a disc or an
elliptical domain.
According to Theorems~\ref{thm:wave} and \ref{thm:means} it is sufficient
to show that
\begin{equation*}
\kl{ \partial_a^2\Ho_a \Ro \chi_\Om}
 \kl{\hat n\kl{x_1,x_0},
\hat a \kl{x_1,x_0}} = 0
\quad \text{ for all  } x_1, x_0 \in \Om \,,
\end{equation*}
where $\hat n\kl{x_1,x_0}$ and $\hat a \kl{x_1,x_0}$
are defined by \req{hatns}.

\subsection{Circular domains}
\label{sec:circulardomain}

We first consider the special case where $\Om = D$ is a disc in the plane and
we assume without loss of generality that $D = \set{x \in \R^2: \abs{x} =1}$ is the
unit disc centered at the origin.

The Radon transform of $\chi_D$ equals
$2 \sqrt{1-a^2} \;  \chi \mset{\sabs{a}^2 < 1}$.
The Hilbert transform of the function $\sqrt{1-a^2}  \; \chi
\mset{ \sabs{a}^2  <  1}$ is known  (see, for example,  \cite[Table 13.11]{Bra00b})
and yields
\begin{equation*}
\kl{\Ho_a  \chi_D } \kl{n , a}
=
2
\begin{cases}
- a - \sqrt{a^2-1}& \text{ if } a < -1 \\\
- a & \text{ if } -1 < s < 1  \\
- a  +  \sqrt{s^2-1}& \text{ if } a > 1
\end{cases} \,.
\end{equation*}
This in particular implies that $\kl{\partial_a^2 \Ho_a  \chi_B} \kl{n , a}
= 0 $ for $\abs{a} < 1$. Because we have  $\hat a\kl{x_1, x_0} < 1 $ for all
$x_1, x_0 \in D $, this implies Theorem \ref{thm:ell} for the case of
circular domains.

\subsection{Elliptical  domains}

Now suppose that  $\Om = E $ is an elliptical domain.
We may assume without loss of generality  that
\begin{equation*}
E = \set{ \kl{x, y} \in \R^2
:  x^2
+ \frac{y^2}{b^2} < 1}
\quad \text{ for some } b > 0 \,.
\end{equation*}
We then have  $\chi_\Om \kl{x,y}  = \chi_D \kl{x, y/b}$,
where $D$ is the unit disc as above. For
any integrable  function $\varphi \colon \R^2 \to \R$ and
any invertible matrix
$B \in \R^{2 \times 2}$ one can easily verify the identity
$\kl{\Ro\varphi^{B}} \kl{n , a}  = \frac{\det \skl{B}}{\sabs{B^Tn }}
\kl{\Ro\varphi} \kl{n/\sabs{B^Tn } , a/\sabs{B^Tn }}$
with $\varphi^{B} \kl{x}  \coloneqq \varphi \kl{B^{-1} x}$.
Consequently, after writing
$\theta\kl{\alpha} = \kl{\cos \alpha, \sin \alpha}$
we obtain
\begin{equation*}
\kl{\Ro  \chi_E} \kl{\theta\kl{\alpha}, a}
=
\frac{b \, \Ro
\chi_D \kl{\frac{ \theta\kl{\alpha}}{ \sqrt{\cos^2 \alpha  + b^2\sin^2 \alpha}}, \frac{a}{\sqrt{\cos^2 \alpha  + b^2\sin^2 \alpha}}}
}{\sqrt{\cos^2 \alpha  + b^2\sin^2 \alpha}}\,
\end{equation*}
and therefore
\begin{equation*}
\kl{\partial_a^2 \Ho_a \Ro  \chi_E}
 \kl{\theta\kl{\alpha}, a}
 =
\frac{b \, \skl{\partial_a^2\Ho_a \Ro  \chi_D }
\kl{\frac{ \theta \kl{\alpha}}{ \sqrt{\cos^2 \alpha  + b^2\sin^2 \alpha}}, \frac{a}{\sqrt{\cos^2 \alpha  + b^2\sin^2 \alpha}}}}{
\kl{\cos^2 \alpha  + b^2\sin^2 \alpha}^{3/2}}
   \,.
\end{equation*}

For $x_0,  x_1 \in \Om$  write $\hat n \kl{x_1, x_0}
= \theta \skl{\hat \alpha} $. Then  we have
$\hat a\kl{x_1, x_0} / \sqrt{\cos^2 \hat \alpha  + a^2\sin^2 \hat \alpha} < 1$.
From Subsection \ref{sec:circulardomain} we already know
that $\kl{\partial_a^2\Ho_a \Ro  \chi_D} \kl{n, a} =0$
for every $\abs{a} < 1$, and therefore obtain
\begin{equation*}
\kl{\partial_a^2 \Ho_a \Ro  \chi_E}
 \kl{\hat n \kl{x_1, x_0}, \hat a\kl{x_1, x_0} }
 = 0
 \quad \text{ for all } x_0 \neq x_1 \in \Om \,,
\end{equation*}
which  establishes  Theorem \ref{thm:ell} for the general
case of  elliptic domains.

\section{Conclusion}
\label{sec:conclusion}

In this paper we derived inversion formulas of the back-projection
type that recover the initial data  of the wave equation from
its solution on the boundary $\partial \Om$ of a convex domain
modulo the smooth term
\begin{equation*}
	\kl{\K_\Om f} \kl{x_0}
	=
	\frac{1}{8\pi}
	\int_{\Omega}
	f\kl{x_1} \, \frac{ \kl{ \partial_a^2 \Ho_a \Ro \chi_\Om} 	\kl{\hat n\kl{x_1,x_0},  \hat a\kl{x_1,x_0}}}
	{\abs{x_1-x_0}}  \, \rmd x_1
\end{equation*}
(see Theorem~\ref{thm:wave}).
In the case of circular and elliptical domains the operator $\K_\Om$  has been shown to vanish identically which yields   to exact reconstruction  formulas (see Theorem~\ref{thm:ell}, Item~\ref{it:ell1}) in these cases.  Corresponding statements have been  derived  for the inversion of the circular mean transform (see Theorem~\ref{thm:means} and   Theorem~\ref{thm:ell}, Item~\ref{it:ell2}).

We note  that all formulas derived in this paper can be implemented in a quite
straight forward manner following the derivations in
\cite{BurBauGruHalPal07,FinHalRak07}.
We do not give details here and refer the interested reader to the papers
\cite{BurBauGruHalPal07,FinHalRak07} for detailed derivations of discrete back-projection type
algorithms. A numerical reconstruction  based on the inversion formula
\req{ell-wave-a} on the  elliptical domain
$E = \sset{\kl{x,y}: x^2 + \kl{y/0.8}^2 < 1 }$ is shown in
Figure~\ref{fig:example}.
It can be seen that except for some smoothing effects at boundaries
(due to the numerical implementation),
the initial data is recovered almost perfectly.

\begin{psfrags}
\begin{figure}
\centering
\includegraphics[height=0.28\textwidth]{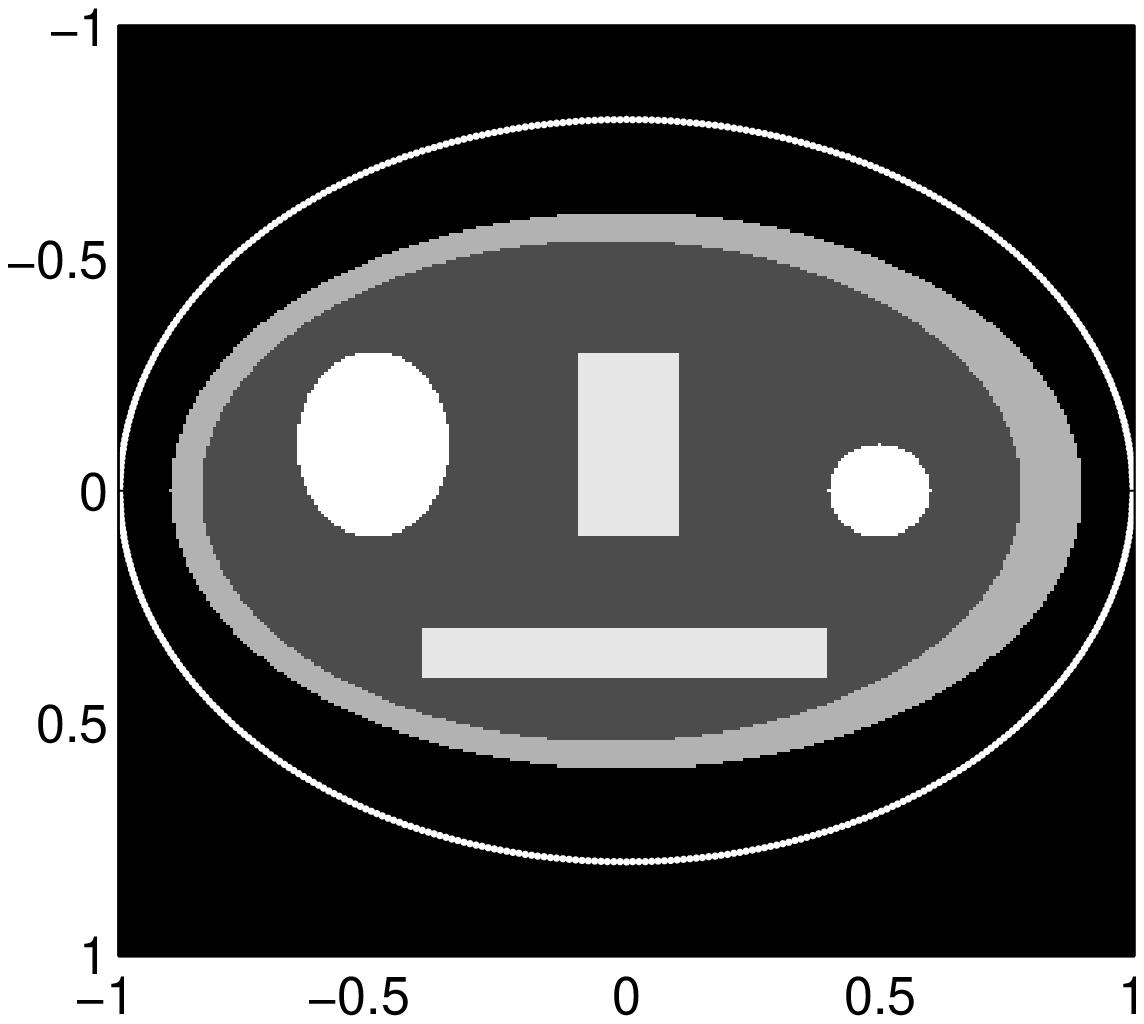}
\includegraphics[height=0.28\textwidth]{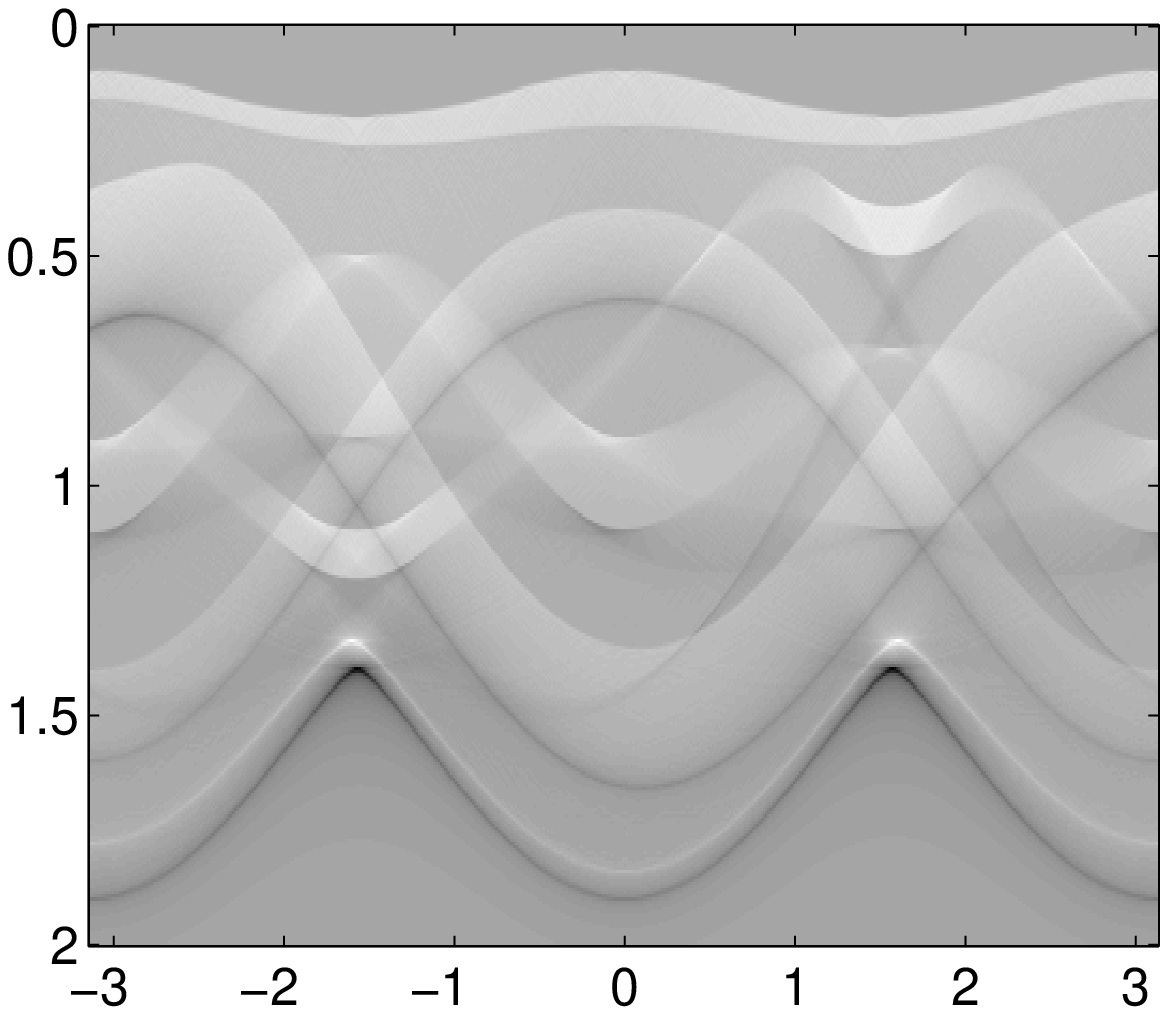}
\includegraphics[height=0.28\textwidth]{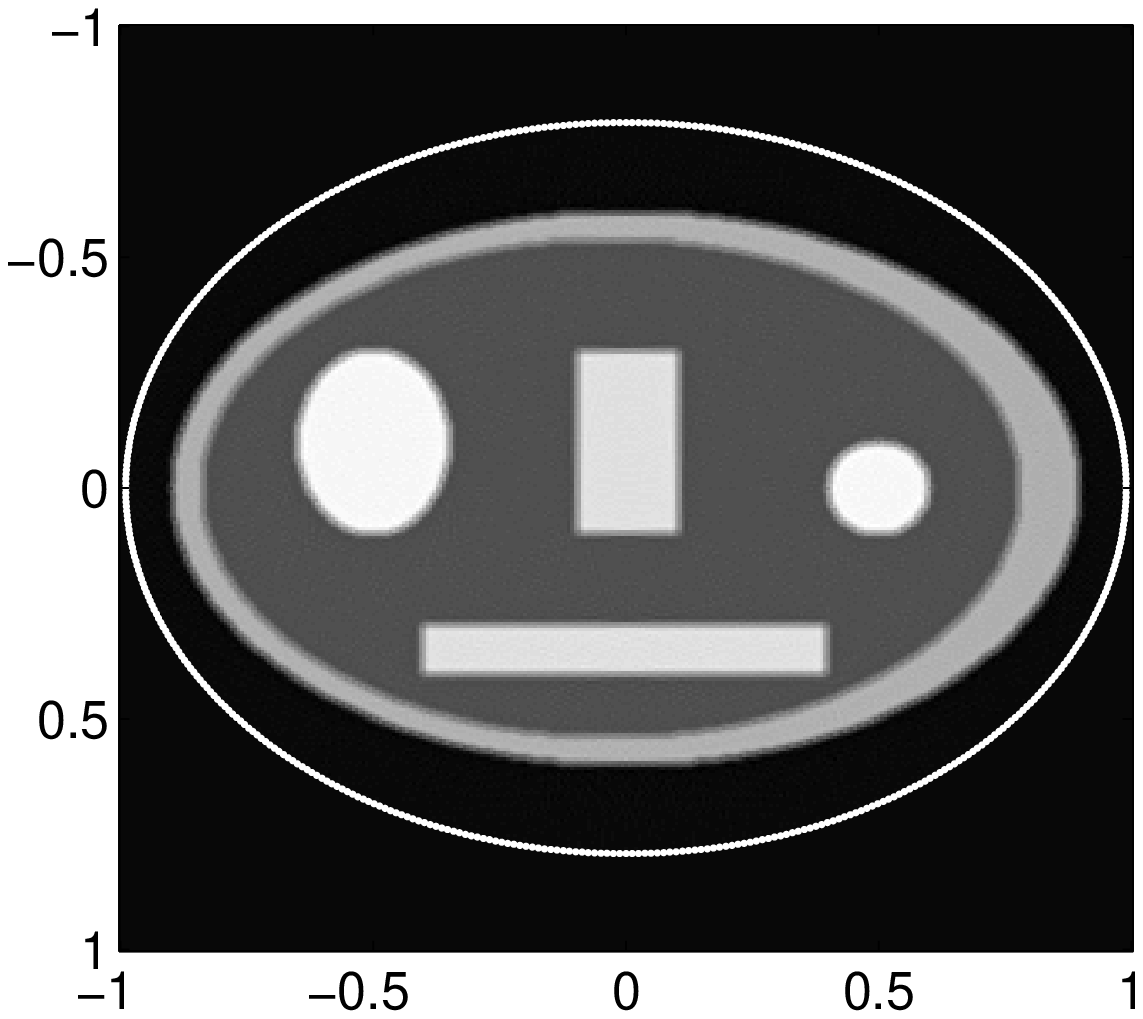}
\caption{Numerical reconstruction  using  inversion formula
\req{ell-wave-a} on the  elliptical domain
$E = \sset{\kl{x,y}: x^2 + \kl{y/0.8}^2 < 1 }$.
Left:  Initial data  $f$.
Middle: Simulated data $\Uo f$.
Right: Numerical reconstruction using inversion formula
\req{ell-wave-a}.
\label{fig:example}}
\end{figure}
\end{psfrags}

With $G$ denoting the diverging  fundamental solution of the two dimensional wave equation and with $u$ denoting the solution of
\req{wave}, our inversion operator
in  \req{inv-wave-a} and   \req{ell-wave-a} can be
written in the form
\begin{equation*}
 	f_{\rm UBP} \kl{x_0}
	=
	 2 \,  \nabla_{x_0}
	\cdot
	\int_{\partial\Om}
	\nu_x
	\int_{0}^\infty
	G\kl{x- x_0,t}
	u \kl{x,t}
	\rmd t
	\rmd s\kl{x}
	\,.
 \end{equation*}
The analog of this expression in three spatial  dimensions (where $G$ and $u$  are replaced by the three dimensional fundamental solution and  the solution of  the three dimensional wave equation, respectively) is the
so-called universal back-projection introduced in the context of photoacoustic tomography in \cite{XuWan05}. In this paper it has also been shown that  the universal  back-projection formula exactly recovers the initial data
of the three dimensional wave equation on spherical domains. This result has been extended
to ellipsoids  in $\R^3$ in \cite{Nat12}. In the latter paper, it is further shown that for any  smooth convex bounded domain $\Omega \subset \R^3$,
the universal back-projection formula recovers the initial data of the three dimensional wave equation modulo the term
\begin{equation*}
	\kl{\K_\Om^{\mathrm{(3D)}}  f} \kl{x_0}
	=
	-\frac{1}{16 \pi^2}
	\int_{\Omega}
	f\kl{x_1} \, \frac{ \kl{ \partial_a^3 \Ro \chi_\Om} 	\kl{\hat n\kl{x_1,x_0},  \hat a\kl{x_1,x_0}}}
	{\sabs{x_1-x_0}^2}  \, \rmd x_1 \,.
\end{equation*}
The identities  derived  in the current  paper demonstrate that
results similar to the ones of  \cite{Nat12}  also holds in two spatial
dimensions.
After initial submission of the present manuscript,
these results have actually been  generalized to arbitrary
dimension (see \cite{Hal12b}).


\def\cprime{$'$} \providecommand{\noopsort}[1]{}

\end{document}